\documentclass{amsart}
\usepackage{verbatim,amssymb,amsmath,amscd,latexsym,amsbsy,mathrsfs} 
\usepackage{graphicx}
\usepackage[all]{xy}
 \input{xy}
 \xyoption{all}

\newtheorem{lemma}{Lemma}[section]
\newtheorem{thm}{Theorem}
\newtheorem{cor}[lemma]{Corollary}

\newtheorem{prop}[lemma]{Proposition}
 \newtheorem{rmk}[lemma]{Remark}

\DeclareMathOperator*{\im}{im} 
 
\DeclareMathOperator*{\rank}{rank} 
\DeclareMathOperator*{\Spec}{Spec}

\DeclareMathOperator*{\DR}{DR }
\DeclareMathOperator*{\Char}{char}

\newcommand{\cX}{\mathcal{X}}

\newcommand{\cD}{\mathcal{D}}
\newcommand{\cE}{\mathcal{E}} 
\newcommand{\cL}{\mathcal{L}}

\newcommand {\V} {\mathcal{V}}

\newcommand {\N} {{\mathbb N}}
\newcommand {\C} {{\mathbb C}}
\newcommand {\R} {{\mathbb R}}
\newcommand {\Z} {{\mathbb Z}}
\newcommand {\Q} {{\mathbb Q}}

\newcommand {\E} {{\mathcal E}}
\newcommand {\dt} {{\bullet}}

\newcommand {\OO} {{\mathcal O}}

\begin{document}
\title[Kodaira-Saito vanishing via Higgs bundles]{Kodaira-Saito vanishing via Higgs bundles in positive characteristic }
\author{
        Donu Arapura    
}
 \thanks {Partially supported by the NSF }
\address{Department of Mathematics\\
 Purdue University\\
 West Lafayette, IN 47907\\
U.S.A.}
\email{arapura@math.purdue.edu}
 \maketitle

 \begin{abstract}
  The goal of this paper is to give a new  proof of a special case of
  the Kodaira-Saito vanishing theorem for a variation of Hodge
  structure on the complement of a divisor with normal crossings. The
  proof does not use the theory of mixed Hodge modules, but instead
  reduces it to a more general vanishing theorem for semistable nilpotent
  Higgs bundles, which is then proved by using some facts about Higgs bundles in positive characteristic.
 \end{abstract}

In 1990, Saito \cite[prop 2.33]{saito} gave a far reaching  generalization of
Kodaira's vanishing theorem using his theory of mixed Hodge modules.
A number of interesting applications have been found in recent years; we
 refer to  Popa's survey \cite{popa} for a discussion of these, and to
 \cite{popa1, schnell} for a discussion of  the theorem itself.  We
 would like  to explain  the  easiest, but still
important, case of the theorem  where the mixed Hodge module is of the form $\R j_*
\V$ (in standard ``non perverse'' notation), where $\V$ is a polarized variation of pure 
Hodge structure on the complement $j:U\to X$ of a divisor with simple normal crossings
$D$. 
Let us also assume that $\V$ has unipotent local monodromies around
components of $D$.  By Deligne \cite{deligne}, the flat vector bundle
$\OO_U\otimes \V$ has a canonical extension  $V$  such 
that the original connection extends to a logarithmic connection
$$\nabla:V\to \Omega_X^1(\log D)\otimes V$$
with nilpotent residues.
By a theorem of Schmid \cite{schmid}, $V$ has a
filtration $F$ by subbundles extending
the Hodge filtration.  This induces a filtration on the de Rham
complex
$$\DR(V,\nabla)=  V \stackrel{\nabla}{\to}\Omega_X^1(\log{D})\otimes V
\stackrel{\nabla}{\to}\Omega_X^2(\log{D})\otimes V\ldots$$
which, for us, starts in degree $0$.
Saito's  theorem tell us that if
$L$ is an ample line bundle and $i>\dim X$, then
$$H^i(X, Gr^\dt_F\DR(V,\nabla)\otimes L)=0$$
More generally, this holds when $\V$ is replaced by an admissible variation of mixed
Hodge structure.

The first goal of this paper is to give a short  proof of this
special case by reduction to characteristic $p>0$. Illusie \cite{illusie} had previously  given a proof  by
such a reduction, when
$V$ arises geometrically from a semistable map of
varieties.  Our proof, however,  is different and it works even for
nongeometric cases (and this seems important for certain
applications, e.g. to Shimura varieties \cite{suh}). We  first replace
the variation of Hodge structure $\V$ by  the vector bundle
$E=Gr_F(V)$, together with the so called Higgs field
$\theta = Gr_F(\nabla)$.
 By work of Simpson, this pair is semistable. 
In addition, the rational Chern classes $c_i(E)=0$, and $\theta$
is nilpotent. In fact, our main result is a vanishing theorem for
 Higgs bundles $(E,\theta)$ satisfying these conditions, regardless of whether or not
 they arise from variations of Hodge structure. We show that
$$H^i(X, \DR(E,\theta)\otimes L)=0$$
for $i>\dim X$, where $\DR(E,\theta)$ is the complex
$\Omega_X^\dt(\log D)\otimes E$ with $\theta$ as the differential.
In order to prove this, we reduce to where the ground field is the algebraic
closure $\overline{\mathbb{F}}_p$ of a finite field with large  characteristic.
We can now avail ourselves of the theory of Higgs bundles in
characteristic $p$ initiated by Ogus and Vologodsky \cite{ov}, and extended to the log setting by
Schepler \cite{schepler}. When combined with work of Langer
\cite{langerD}, we get an
operation $B$ from the class of Higgs bundles, satisfying the previous
conditions, to itself
that satisfies a ``bootstrapping''
inequality 
$$\dim H^i(\DR(E,\theta)\otimes L)\le 
\dim H^i(\DR(B(E,\theta))\otimes L^p)$$
By iterating $B$, we get a sequence of Higgs
bundles $(E_j,\theta_j)=B^j(E,\theta)$.
(Such sequences were first considered -- for different reasons  --  by
Lan, Sheng and  Zuo \cite{lsz1}; they call them Higgs-de Rham sequences.)
 Langer \cite{langerI}, and also Lan, Shen, Yang, Zuo \cite{lsyz}, show  that this sequence is
eventually periodic, and this is the place where we need to work over
$\overline{\mathbb{F}}_p$. 
 These facts together with  Serre vanishing implies the theorem.
Replacing   Serre vanishing by methods  developed by the author \cite{arapura} yields  a
stronger result. Suppose $M$ is a vector bundle such that for some
integer $m>0$
and some effective divisor $D'$ supported on $D$ with coefficients
less than $m$, $S^m(M)(-D')$ is ample; or more succinctly, suppose
that $M(- \Delta)$ is ample  for some fractional $\Q$-divisor $\Delta$ supported on $D$  (e.g. $\Delta=\frac{1}{m}D'$).
Then
$$H^i(X, Gr_F\DR(V,\nabla)\otimes M)=0$$
for $i\ge\dim X+\rank M$. When $M$ is a line bundle, this due to Suh
\cite{suh}, who deduced it from Saito's theory; also see \cite{wu}. 

As an application of the first vanishing theorem, we prove a Fujita
type result  that given a complex variation of Hodge structure $\V$ with the
same assumptions as above,
$\wedge^a Gr^b_F(V)\otimes \det(F^{b+1})$ is numerically semipositve for every $a$
and $b$ (we
define $\det 0 =\OO_X$).  From the strengthened form of the vanishing
theorem stated at the end of the previous paragraph,
  we deduce an extension of the Koll\'ar-Saito vanishing
theorem. Suppose that $Y$ is an arbitrary complex projective variety 
and $V$ is a variation of Hodge structure supported on a smooth
Zariski open of $Y$, with quasiunipotent monodromy at infinity. Then
there is a natural  extension $S(V)$ of $F^{\max }V$
to a coherent sheaf  on $Y$. If $M$ is an ample vector bundle on $Y$ then
$$H^i(Y, S(V)\otimes M)=0$$
for  $i\ge \dim Y+\rank M$.

My thanks to Yohan Brunebarbe, Adrian Langer, Mihnea Popa, and Kang
Zuo for various comments, and to Junecue Suh for sending me
\cite{suh}.

\section{Statement of the main results}

In order to put the main results in context,
 let us  quickly summarize the relevant   ideas of Simpson \cite{simpson}, and their extensions due
to Mochizuki \cite{mochizuki, mochizuki2}. We wish to point out from the beginning
that since we will be applying these to a
variation of Hodge structure, we will not need any of  the deeper existence
theorems from these papers. The only hard result needed is Simpson's
semistability theorem explained below.
Let $X$ be a $d$ dimensional smooth  projective variety defined over an algebraically
closed field $k$.  Fix a reduced effective divisor with simple normal
crossings $D$ and  an ample line bundle $L$ on $X$.
These assumptions will hold throughout
the paper. We also assume for the next few paragraphs that the ground
field $k=\C$.

We may identify the symmetric space $ Gl(r)/U(r)$ with the space of
positive definite hermitian matrices by sending $M\in GL(r)$ to $ M^*M$. Therefore
given a local system $\mathcal{V}$ on $(X-D)^{an}$ of
rank say $r$, a  hermitian metric $h$ on it can be viewed as a $\pi_1((X-D)^{an})$-equivariant
$C^\infty$ map from the universal cover $\widetilde{(X-D)^{an}}$ to 
$ GL(r)/U(r)$. The pair $(\mathcal{V}, h)$ is called a
tame harmonic bundle if $h$, viewed as a map as above, is harmonic with mild
singularities near $D$. Simpson (when $d=1$) and Mochizuki  ($d$
arbitrary) show how to associate to any tame
harmonic bundle, a so called parabolic or filtered  Higgs bundle. This consists of
 a holomorphic  vector bundle $E$ on $X$, an $\OO_X$-linear map
$$\theta:E\to \Omega_X^1(\log{D})\otimes E$$
such that  $\theta^2=\theta\wedge\theta$ viewed as a section of  $H^0(
\Omega_X^2(\log{D})\otimes E)$ is zero, and certain filtrations along
$D$.  We will not need to make the last part precise, since  these filtrations will be
trivial for the cases of interest to us that will be described in the next paragraph.

For us, the key example of a harmonic bundle  arises as follows.
A polarized complex variation of Hodge structure of weight $n$ on $X-D$ is a local
system $\mathcal{V}$ on $(X-D)^{an}$ with an indefinite form $Q$ (the polarization) and  a 
bigrading of the associated $C^\infty$ vector bundle
$$C_X^\infty\otimes_\C \mathcal{V} = \bigoplus_{p+q=n} V^{pq}$$
These  are required to  satisfy Griffiths transversality
(appropriately formulated) and the
Hodge-Riemann relations. However, unlike the
usual notion, there is no requirement about the existence of  a $\Z$ or
$\Q$ lattice in $\mathcal{V}$.  Let us also suppose that the monodromies of the local system about
components of $D$ are unipotent.  This assumption simplifies the
story, and in the geometric case, it is close to automatic; more precisely
it can always be achieved by pulling back to a
branched cover.  After adjusting signs, the polarization
determines a positive definite metric $h=\sum (-1)^pQ|_{V^{pq}}$ on
$\mathcal{V}$ such that the pair is tame harmonic.
The associated  Higgs bundle  is  given explicitly by $E\cong Gr_F(V)$ with Higgs field
$$\theta=Gr_F(\nabla):E\to  \Omega_X^1(\log{D})\otimes E$$
 induced by the connection. The unipotency of  local monodromies,
forces the  parabolic structure to be   trivial.
For our purposes, a Higgs bundle on $(X,D)$ will simply refer to a
pair $(E,\theta)$ as above without any parabolic structure.

Suppose that   $(E,\theta)$  is a Higgs bundle arising from a complex
 variation of Hodge structure with unipotent local monodromy.
Then it has a number of special features that we need to explain.

\begin{enumerate}
\item The Chern classes of $E$, in rational cohomology, all
vanish. This is  because $c_i(V)=0$ by \cite[cor
B3]{ev} and $c_i(V)=c_i(E)$ since $V$ and $E$ have the same class in
the Grothendieck group.
\item The Higgs bundle is semistable in the sense that $\mu(E')\le \mu(E)=0$ 
for any proper coherent 
subsheaf $E'\subset E$ stable under $\theta$, where 
$$\mu(E)=\frac{\deg E}{\rank E} = \frac{c_1(E)\cdot L^{d-1}}{\rank
  E}$$
 is the slope.  We can argue as follows. After replacing $L$ by a
 power, we can assume that it is very ample. Let $C\subset X$ be a curve given as a complete
 intersection of general divisors in $|L|$. Then it is enough to show
 that $\mu(E'|_C)\le 0$. So it suffices to prove the semistability of
 $(E,\theta)$ after replacing $X$ by $C$, $\V$ by  $\V|_C$ etc.  For
 curves  semistability is proved by Simpson \cite[thm
 5]{simpson}. Simpson proves this more generally for Higgs bundles 
 arising from tame harmonic bundles. In this generality, the parabolic
 structure may be nontrivial so the 
 definition of semistability is a bit more complicated; however, it reduces
 to what we gave when the parabolic structure is trivial, as is the
 case here.

\item Finally observe that our bundle carries a grading $E=\bigoplus E^p$ such that
$\theta(E^p)\subset E^{p-1}$. This implies that $\theta$ is nilpotent
in the sense that $E$ carries a filtration $E=N^0\supset N^1\ldots$  such that it stable under
that action of $\theta$, and the induced action on the associated
graded is zero. The length of the shortest such filtration is called
the level of nilpotence of $\theta$.

\end{enumerate}

Higgs bundles of this type also arise from certain variations of
mixed Hodge structures called admissible variations. We will not recall the precise axioms, which
are rather technical, but instead  we use a weaker  notion. 
Let us say that a weak  complex variation
of mixed Hodge structures with unipotent local mondromies around $D$, consists of local system $\V$ on 
$(X-D)^{an}$ with unipotent local monodromy, a filtration $\mathcal{W}\subset \V$ by sublocal systems,
a filtration of the Deligne extension $V$ of $\V\otimes \OO_{X-D}$ by subbundles  $F^\dt\subset V$ satisfying Griffiths transversality 
$$\nabla(F^p)\subset \Omega_X^1(\log D)\otimes F^{p-1}$$
We require that the associated graded with respect to $W$ is a polarizable pure variation of Hodge structure.
We  can see that $Gr_F(V)$ is again a Higgs bundle, with a filtration by Higgs bundles induced by $\mathcal{W}$.

\begin{lemma}
 An extension of two nilpotent, semistable Higgs bundles with vanishing Chern classes has the same property.
\end{lemma}

\begin{proof}
  The nilpotence and vanishing of Chern classes are immediate from the definition and the Whitney sum formula. 
Given an exact sequence 
$$0\to (E_1, \theta_1)\to (E,\theta)\to (E_2,\theta_2)\to 0$$
of Higgs bundles and a sub Higgs sheaf $(E',\theta') \subseteq (E,\theta)$
$$\deg E'= \deg E'\cap E_1 + \deg \im (E'\to E_2)\le 0$$ 
\end{proof}

\begin{cor}
  The Higgs bundle $Gr_F(V)$ associated to a weak complex variation of mixed Hodge structures is
 nilpotent, and semistable with vanishing Chern classes.
\end{cor}

Given a Higgs bundle $(E,\theta)$ on $(X,D)$, since $\theta^2=0$, we
get a ``de Rham'' complex
$$\DR(E,\theta) = E \stackrel{\theta}{\to} \Omega_X^1(\log{D})\otimes E \stackrel{\theta}{\to}
\Omega_X^2(\log{D})\otimes E\ldots$$
If this arises from a  weak complex variation of mixed complex of Hodge structure $\V$ as
above, $\DR(E,\theta)$ can be identified with the assocated
graded $Gr_F\DR(V,\nabla)$ of the de Rham complex with respect to the
filtration
$$F^p \DR(V,\nabla) = F^pV \to\Omega_X^1(\log{D})\otimes
F^{p-1}V\ldots$$
Finally, observe that the notions of Higgs bundle, semistability,
nilpotence and $\DR$ make sense over any field
$k$. Following \cite{di}, we say that the pair $(X,D)$  is liftable modulo
$p^2$ if it lifts to a smooth scheme over $\Spec W_2(k)$ with a relative
normal crossing divisor.

\begin{thm}\label{thm:main1}
Let $(X,D,L)$, as above, be defined over an algebraically closed field
$k$. Let $(E,\theta)$ be a nilpotent semistable Higgs bundle on $(X,D)$ with
vanishing Chern classes in $H^*(X_{et}, \Q_\ell)$. Suppose that either
\begin{enumerate}
\item[(a)] $\Char k= 0$, or
\item[(b)] $\Char k=p>0$, $(X,D)$ is liftable modulo $p^2$,  $d+ \rank
  E< p$.
\end{enumerate}
Then
$$H^i(X,\DR(E,\theta)\otimes L)=0$$
for $i>d$.
\end{thm}

We should remark that we only really need to assume that $c_1(E)=0$ and $c_2(E)\cdot
L^{d-2}=0$.

\begin{cor}[Saito]\label{cor:saito}
  If $\V$ is a weak complex variation  of Hodge mixed structure on $(X,D)$ with
  unipotent local monodromies around $D$, then
$$H^i(X,Gr_F\DR(V,\nabla)\otimes L)=0$$
for $i>d$.
\end{cor}

The next result is a refinement of the semipositivity results of Fujita
\cite{fujita},  Kawamata \cite{kawamata} and Peters \cite{peters}.
It also overlaps with extensions due to Fujino, Fujisawa and Saito
\cite{ffs} and Brunebarbe \cite{brunebarbe}; see remark
\ref{rmk:brune} for further discussion.
Recall that a vector bundle on $X$ is nef, or numerically
semipositive, if any quotient line bundle of the pullback of $E$ to a curve has
nonnegative degree. Numerous other characterizations can be found in
\cite{lazarsfeld}. As usual, we let $\det(E)=\wedge^{\rank E} E$ if $E$ is a
nonzero vector bundle. For notational convenience, we define $\det(0)=\OO_X$.

\begin{thm}\label{thm:semipos}
  Suppose that $V$ is a weak complex variation of mixed   Hodge structure on $(X,D)$ with
  unipotent local monodromies around $D$. Then  $\wedge^a Gr_F^b(V)\otimes
  \det(F^{b+1})$ is nef for every $a$ and $b$. 
\end{thm}

\begin{cor}
  Let $F^{\max}\subset V$ be the smallest
nonzero Hodge bundle, then  $F^{\max}$ is nef. For every $b$,
$\det F^b$ is nef.
\end{cor}

Suh \cite{suh} considered a generalization of Saito's theorem where the ampleness of $L$
is relaxed. We will prove a version of this as well.
Let $\Delta$ be a $\Q$-divisor  and let $m$ be the least common
multiple of the denominators of the coefficients of $\Delta$. We will
say that it is effective and fractional if the coefficients lie in $[0,1)$.
Given a vector bundle $M$, let us say that
{\em $M(-\Delta)$ is ample} if $S^m(M)(-m\Delta)$ is ample in the
usual sense.

\begin{thm}\label{thm:main2}
  Make the same assumptions as in theorem \ref{thm:main1}.
Let  $M$ be a vector bundle when $\Char k=0$, or a line bundle when
$\Char k=p$, such that  $M(-\Delta)$ is ample
for some fractional  effective $\Q$-divisor supported on $D$.  
Then 
\begin{equation*}
H^i(X,\DR(E,\theta)\otimes M(-D))=0  
\end{equation*}
for $i\ge d+ \rank M$.
\end{thm}

While clearly the
last theorem implies the first, we prefer to state and
prove them separately. The  proof of the first is somewhat
easier and it serves as a model for the last.

From the previous theorem, we  can  deduce a Koll\'ar-Saito type vanishing theorem for 
arbitrary varieties.
Let $Y$ be possibly singular complex projective variety, and let 
$\mathcal{V}$ be a complex variation of Hodge structures
   on  smooth Zariski open $U\subset Y$ with quasiunipotent
   monodromies at infinity. The last condition means that every
   restriction of the pull back of $\mathcal{V}$ to a punctured disk
   has quasiunipotent monodromy. For example, quasiunipotence holds
   when $\mathcal{V}$ is of geometric origin.
Choose a  desingularization
 $\pi: X\to Y$ which is an isomorphism over 
 $U$ and  such that the preimage $D=\pi^{-1}(Y-U)$ has simple
 normal crossings. Let $j:U\to X$ denote the inclusion.
    Let 
$$\nabla:V\to \Omega_X^1(\log D)\otimes V$$
be the Deligne extension of $\OO_U\otimes \mathcal{V}$ such that
eigenvalues of the residues lie in $(-1,0]$. We define 
$S'(V) =\omega_X\otimes (V\cap j_*F^{\max}(\OO_U\otimes \mathcal{V}))$ and $S(V)=\pi_*S'(V)$.
This coincides  with the similarly named object defined by Saito  in
\cite{saito2} because of [loc. cit., (3.1.1), thm 3.2].
Consequently, $S(V)$ is independent of the choice of $X$, although
this is easy to check directly as well.

\begin{thm}\label{thm:ks}
Let us keep the  notation  and assumptions of the previous paragraph.
Let $M$ be an ample vector bundle on $Y$, then
\begin{enumerate}
\item[(a)] 
$H^i(X,  S'(V)\otimes \pi^*M)=0$ 
for $i\ge \dim X +\rank M$.
\item[(b)] $R^i\pi_* (S'(V))=0$ for $i>0$.
\item[(c)] 
$H^i(Y, S( V)\otimes M)=0$
for $i\ge \dim X +\rank M$.
\end{enumerate}

\end{thm}

When $M$ is a line bundle, this is due to Koll\'ar \cite{kollar} when
$V$ is of geometric origin and
 Saito \cite{saito2} in general.

\section{Proof of  theorem \ref{thm:main1}}

Before giving the proof, we need to explain
the basic tools. Let $(X,D,L)$ be as in the previous section, but now
defined over an
arbitrary  algebraically closed field $k$. 

\begin{lemma}\label{lemma:serre}
  If $L$ is an ample line bundle, then
$$H^i(X,\DR(E,\theta)\otimes L^{N})=0$$
for $i>d$ and $N\gg 0$.
\end{lemma}

\begin{proof}
This follows from Serre vanishing and the spectral sequence
$$E_1^{ab} =H^b(\Omega_X^a(\log D)\otimes E\otimes L^N)\Rightarrow H^{a+b}(  \DR(E,\theta)\otimes L^{N})$$
\end{proof}
 
A vector bundle with an integrable logarithmic connection 
$$\nabla:V\to\Omega_X^1(\log D)\otimes V$$
defined in the usual way, will simply be  referred   as a flat bundle below. If $(V, \nabla)$ is a flat bundle and $F^\dt $  a filtration on $V$ by subbundles satisfying
Griffiths transversality $\nabla(F^p)\subset F^{p-1}$, then $(Gr_F(V),
Gr_F(\nabla))$ is a Higgs bundle exactly as above.
We say that that the connection $(V,\nabla)$  is semistable if $\mu(V')\le
\mu(V)$ for every $\nabla$-stable saturated subsheaf $V'\subseteq
V$. (Saturation means that $V/V'$ is torsion free.) We have the following
fact:

\begin{thm}[Langer {\cite[thm 5.5]{langerD}}]\label{thm:langer}
  If $(V,\nabla)$ is a semistable flat bundle, there exists a
  canonical filtration $V=F_{can}^0 V\supseteq F_{can}^1V\ldots$ satisfying Griffiths
  transversality such that $\Lambda(V,\nabla)=(Gr_{F_{can}}(V),
Gr_{F_{can}}(\nabla))$ is a semistable Higgs bundle.
\end{thm}

The second ingredient is the positive  characteristic version of Higgs
bundle theory due to  Ogus-Vologodsky \cite{ov} and Schepler
\cite{schepler}.  Assume now that  the characteristic is $p>0$.
Let $Fr:X\to X$ be the absolute Frobenius
given by the identity on the space and $p$th power on the structure
sheaf.
This factors as 
$$X\stackrel{Fr'}{\longrightarrow} X' \stackrel{\Phi}{\to} X$$
where the  map $\Phi$ is base change along the absolute Frobenius of
$k$. The map $\Phi$ is an  isomorphism of schemes because $k$ is perfect.
The first map $Fr'$, called the relative Frobenius,  is $k$-linear and is
what is used is \cite{ov}. It is more convenient for us to work with
$X$ alone  and transport objects from $X'$ to $X$ via $\Phi$ when necessary.

We recall some basic facts about flat bundles in  positive
characteristic, and refer to \cite{katz} for more details.
 We have a new invariant for  a flat bundle $(V,\nabla)$ called the
$p$-curvature, which gives the obstruction for $\nabla$ to commute with
$p$th powers. We denote the pullback $Fr^*M$ of a vector
bundle by $M^{(p)}$. This  carries a canonical
integrable connection $\nabla_{cart}:M^{(p)}\to \Omega_X^1\otimes M^{(p)}$ with trivial $p$-curvature, that we call the Cartier
connection.  Under the natural embedding $M\hookrightarrow
M^{(p)}=\OO_X\otimes_{\OO_{X'}}\Phi^*M$ given by $m\mapsto 1\otimes m$,
we have $\ker \nabla_{cart}\cong M$. 

\begin{lemma}\label{lemma:DRcart}
 Let $M$ be a vector bundle and let $(V,\nabla)$ be a flat
 bundle on $(X,D)$. Equip  $V\otimes M^{(p)}$ with the tensor product connection
 $\nabla_T=\nabla\otimes \nabla_{cart}$. Then
$$(Fr_*\DR(V,\nabla))\otimes M\cong
Fr_*\DR(V\otimes M^{(p)}, \nabla_T)
$$
as complexes of $\OO_X$-modules.
\end{lemma}

\begin{proof}
  The projection formula yields isomorphisms
$$ \pi:Fr_*(\Omega_X^i(\log D)\otimes V)\otimes M\cong
Fr_*(\Omega_X^i(\log D)\otimes V\otimes M^{(p)}) $$
If $\alpha$ and $m$ are local sections of  $Fr_*(\Omega_X^i(\log
D)\otimes V)$ and $M$, then $\nabla_T(\pi(\alpha\otimes m)) =
\pi(\nabla(\alpha)\otimes m)$ because $\nabla_{cart}m=0$.
\end{proof}

  A flat bundle is called nilpotent (of level at most $n$) if the there
exists a filtration (of length at most $n$) by flat bundles such
that the $p$-curvature of the associated graded vanishes.
To avoid confusion, we should point out that the word ``nilpotent''
will be used in three different senses in the next theorem.  Two have
already been explained. We say that the residue
$Res_{D_i}(\nabla)\in End(V|_{D_i})$ along a component $D_i\subset D$ is
nilpotent of level $\le \ell$ if $(Res_{D_i}\nabla)^\ell=0$.

\begin{thm}[Ogus-Vologodsky, Schepler]\label{thm:osv1}
   If $(X,D)$ lifts modulo $p^2$, then there exists an equivalence between
the category of flat bundles
  which are nilpotent of level less than $p$ and with residues
  nilpotent of level  $\le p$, and the category of  Higgs
  bundles which are nilpotent of level at most $p$. The functor giving the
  equivalence, which depends on the choice of lifting, is called the Cartier transform $C$.
\end{thm}

\begin{proof}
The general result follows from \cite[cor 4.11]{schepler}. When
$D=\emptyset$, this was first proved in \cite[thm 2.8]{ov}.
  A simplified construction of the
correspondence in the last case can be found in \cite{lsz}.
\end{proof}

\begin{thm}[Ogus-Vologodsky, Schepler]\label{thm:osv2}
  Suppose that $(X,D)$ lifts modulo $p^2$,  $(V,\nabla)$ is  a flat bundle
  which is nilpotent of level $\ell$ and with residues
  nilpotent of level $\le p$.
 Let
  $(E,\theta)=C(V,\nabla)$. If $\ell + d<p$, there is an isomorphism
$$ Fr_*\DR(V,\nabla)\cong \DR(E,\theta) $$
in the derived category.
\end{thm}

\begin{proof}
This is \cite[cor 5.7]{schepler}, and \cite[cor 2.27] {ov} in the
non-log case. 
\end{proof}

\begin{cor}\label{cor:osv0}
If $(E,\theta)$ is a nilpotent Higgs bundle of level less than $p-d$, then
for any vector bundle $M$, we have
\begin{equation*}
 H^i(X, \DR(E,\theta)\otimes M) =
H^i(X,\DR(C^{-1}(E,\theta)\otimes M^{(p)}))
\end{equation*}  
In the above, $M^{(p)}$ is equipped with $\nabla_{cart}$ and  
$C^{-1}(E,\theta)\otimes M^{(p)}$ should be understood as the
tensor product of bundles with connection.
\end{cor}

\begin{proof}
The theorem together with lemma~\ref{lemma:DRcart} and finiteness of
$Fr$ shows that
\begin{equation*}
  \begin{split}
  H^i(X, \DR(E,\theta)\otimes M) &= H^i(X,
  Fr_*(\DR(C^{-1}(E,\theta)\otimes M^{(p)})))\\
&= H^i(X,\DR(C^{-1}(E,\theta)\otimes M^{(p)}))
\end{split}
\end{equation*}  

\end{proof}

Let $h^i(-)$ denote the  dimension $\dim H^i(-)$ below.

\begin{cor}\label{cor:osv}
If $(E,\theta)$ is a nilpotent Higgs bundle of level less than $p-d$, then
for any vector bundle $M$, we have
\begin{equation}\label{eq:osvM}
h^i(X, \DR(E,\theta)\otimes M) \le
h^i(X, \DR( \Lambda C^{-1}(E,\theta))\otimes M^{(p)}) 
\end{equation}  
In particular, if $L$ is a line bundle then
\begin{equation}\label{eq:osvL}
h^i(X, \DR(E,\theta)\otimes L) \le
h^i(X, \DR( \Lambda C^{-1}(E,\theta))\otimes L^{p}) 
\end{equation}  
\end{cor}

\begin{proof}
By corollary \ref{cor:osv0}, for \eqref{eq:osvM}  it is enough to prove that 
$$h^i(X,\DR(C^{-1}(E,\theta)\otimes M^{(p)}))\le h^i(X, \DR( \Lambda C^{-1}(E,\theta))\otimes M^{(p)}) $$
This  follows from the next lemma \ref{lemma:Lambda}. For \eqref{eq:osvL}, we use the isomorphism $L^{(p)}\cong L^p$.
\end{proof}

\begin{lemma}\label{lemma:Lambda}
If $(V,\nabla)$ is a semistable flat bundle on $(X,D)$ and $N$ another 
vector bundle,  we have
  $h^i(X,\DR(V,\nabla)\otimes N)\le h^i(X, \DR( \Lambda (V,\nabla)\otimes N)) $
\end{lemma}

\begin{proof}
The spectral sequence abutting to $H^*(\DR(V,\nabla)\otimes N)$ associated to the filtration
$F^\dt_{can}\DR(V,\nabla)\otimes N$ has
$$\bigoplus_{a+b=i}E_1^{ab} = H^i(X,\DR(\Lambda(V,\nabla))\otimes N)$$
The lemma is now a consequence of  the standard inequality $\dim E_1^{ab}\ge \dim E_\infty^{ab}$.
\end{proof}

\begin{rmk}\label{rmk:osv}
  Since the level is bounded by the rank, the previous corollary
  applies when $d+\rank E <p$.
\end{rmk}

Set $B(E,\theta) = \Lambda C^{-1}(E,\theta)$. This operator, which will give a map from
the set of Higgs bundles of the appropriate type to itself, was first
introduced implicitly  in \cite{lsz1}.

\begin{thm}[Langer]\label{thm:langer2}
  Assume that $(X,D)$ lifts modulo $p^2$.
  \begin{enumerate}
  \item Suppose that $(E,\theta)$ is a Higgs bundle nilpotent of level
    at most $p$. Then $(E,\theta)$ is semistable if and only if $C^{-1}(E,\theta)$
    is semistable.
\item Suppose that $k=\overline{\mathbb{F}}_p$ is the algebraic closure of a
  finite field. If $(E,\theta)$ is semistable with vanishing
  $\ell$-adic  Chern classes and rank $\le p$, then the sequence
  $(E_i, \theta_i)= B^i(E,\theta)$ is eventually periodic in
  the sense that $(E_n,\theta_n) = (E_m,\theta_m)$ for some $m>n$.
  \end{enumerate}
\end{thm}

\begin{proof}
  The first item is \cite[cor 5.10]{langerD}. The second is \cite[prop
  1]{langerI} together with the remarks at the beginning of section
  3.1 of \cite{langerL}. When $D=\emptyset$, a different proof of (2) can
  be found in \cite{lsyz}.
\end{proof}

For the reader's convenience, we restate theorem \ref{thm:main1}.

\newcounter{thmcount}
\setcounter{thmcount}{\value{thm}}
\setcounter{thm}{0}

\begin{thm}
Let $(X,D,L)$, as above, be defined over an algebraically closed field
$k$. Let $(E,\theta)$ be a nilpotent semistable Higgs bundle on $(X,D)$ with
vanishing Chern classes in $H^*(X_{et}, \Q_\ell)$. Suppose that either
\begin{enumerate}
\item[(a)] $\Char k= 0$, or
\item[(b)] $\Char k=p>0$, $(X,D)$ is liftable modulo $p^2$,  $d+\rank
  E< p$.
\end{enumerate}
Then
$$H^i(X,\DR(E,\theta)\otimes L)=0$$
for $i>d$.
\end{thm}

\begin{proof}
 Suppose that we are in the case (a) where $\Char k=0$.  We choose a finitely generated subring $A\subset k$, such that  $X,
  D, E,\theta, L$ are all defined over it. In other words, we have a
  projective scheme $\cX\to \Spec A=S$ such that the geometric generic
  fibre of  $\cX$ is $X$, and  objects  with $\cD, \cE, \Theta,
  \cL$ over $\cX$ which restrict to $D$ etc.. After  shrinking
  $S$ if necessary, we can assume that 
  \begin{enumerate}
\item The characteristics of the residue fields of closed points of
  $S$ are bigger than $d+ \rank E$,
\item $\cX\to S$ is smooth, and $\cD$ is a divisor with relative normal crossings,
\item $\cL$ is a relatively ample line bundle,
  \item $\cE$ is flat over $S$, and $\Theta$ has constant rank
\item The restrictions of $(\E,\Theta)$ to the geometric fibres are
  semistable (by the openness of semistability, c.f. \cite[thm 1.14]{maruyama}),
  \end{enumerate}

In  case (b), $\Char k=p>d+\rank E$ and we have a lifting of $(X,D)$ over
$W_2(k)$. We choose a finitely generated subring $A\subset W_2(k)$, so
that $(X,D)$ is defined over it, and such that  the remaining objects
$L,\ldots$ are defined over
$A/(p)$. Let $S=\Spec A/(p)$, and $\cX\to S$, $\cD\to S$ be the
correspond families. The pair lifts to a pair $(\tilde \cX,\tilde\cD)$ 
over $\tilde S=\Spec A$.
 We may shrink $\tilde S$ so that 
$(\tilde X,\tilde\cD)/\tilde S$ is  smooth with relative normal crossings
and such that the above assumptions (3)-(5) hold. 

From this point on, we will   treat both cases in parallel.
Let $\cL_s,\cL_{\bar s},\ldots $ denote the restrictions of these
object to the fibre $\cX_s$ and geometric fibre $\cX_{\bar s}=\cX\times_S\Spec
\overline{k(s)}$.
Then by semicontinuity of cohomology, it is enough to prove that
$$H^i(\cX_{\bar s},\DR(\cE_{\bar s},\Theta_{\bar s})\otimes \cL_{\bar
  s})=H^i(\cX_s,\DR(\cE_s,\Theta_s)\otimes \cL_s)\otimes\overline{k(s)}=0$$
for all closed points $s\in S$. Now fix such an $s$. The residue field $k(s)$
is finite of characteristic $p>d+\rank E$. 
By theorem \ref{thm:langer2}, corollary \ref{cor:osv} and remark \ref{rmk:osv}, there is
 an eventually periodic sequence of Higgs bundle $(E_j,\theta_j)$ such
 that $(E_0,\theta_0)= (\cE_{\bar s}, \Theta_{\bar s})$ and
$$h^i(\DR(E_0,\theta_0)\otimes L_{\bar s})\le 
h^i(\DR(E_1,\theta_1)\otimes L_{\bar s}^p)\le h^i(\DR(E_2,\theta_2)\otimes
L_{\bar s}^{p^2})\ldots$$
Since the sequence of bundles is eventually periodic, we can bound the
initial term by 
$$h^i(\DR(E_n,\theta_n)\otimes L_{\bar s}^{p^j})$$
for some $n$ and $j$ arbitrarily large. This  is zero by lemma \ref{lemma:serre}, and this completes the proof.
\end{proof}

\begin{rmk}\label{rmk:langer} Langer pointed out to the author that the proof of case
  (b) can be simplified slightly. It is not necessary to reduce to
  finite fields. Instead, arguing as in the proof of \cite[prop  1]{langerI}, one 
sees that   the set of Higgs bundle $(E,\theta)$, satisfying the above
conditions, forms a bounded family. Therefore there is a uniform
constant $N_0$ such  that $H^i(\DR(E,\theta)\otimes L^N) =0$ for $N\ge
N_0$ and all $(E,\theta)$.  Now the result follows from
corollary \ref{cor:osv}.
\end{rmk}
\section{Proof of theorem~\ref{thm:semipos}}

\begin{lemma}\label{lemma:semipos}
  Let $E$ be a vector bundle and $K,L$  line bundles on a projective variety $Y$ with $L$ very
   ample.   Suppose that 
$$H^i(S^n(E)\otimes K\otimes L^m)=0$$
for all $i>0$, $n\gg 0$ and $m\gg 0$. Then $E$ is nef.
\end{lemma}

\begin{proof} Choose $m> \dim Y +1$. The
  hypothesis implies that for all $n$, $S^n(E)\otimes K\otimes L^m$ is $0$-regular
  in the sense of Castelnuovo-Mumford, and therefore  globally
  generated \cite[\S 1.8]{lazarsfeld}. Choosing $m$ large enough so that $K\otimes
  L^m$ is ample allows us to apply \cite[6.2.13]{lazarsfeld} to conclude
  that $E$ is nef.
\end{proof}

Suppose $V$ and $W$ are weak complex variations of mixed Hodge
structures on $X-D$ with unipotent monodromy around $D$. Then $V\otimes
W$ also carries a
variation of Hodge structure with unipotent monodromy around $D$.
Let $F^\dt V$ and $F^\dt W$ denote the Hodge filtrations of $V$ and
$W$ respectively. Then 
$$F^c(V\otimes W) =\bigoplus_{a+b=c} F^a  V\otimes F^bW$$
is the Hodge filtration on the tensor product. It follows easily that
$$Gr_F^c(V\otimes W) =\bigoplus_{a+b=c} Gr_F^a  V\otimes Gr_F^bW$$
One can deduce the
corresponding formulas for the symmetric powers
\begin{equation}
  \label{eq:Sn}
Gr_F^{b} (S^n V) = \sum_{n_1+\ldots n_k=n\atop n_1b_1+\ldots n_kb_k=b} \im\left[ S^{n_1}(Gr_F^{b_1} V)\otimes  \ldots\otimes S^{n_k}(Gr_F^{b_k} V)\right]
\end{equation}
and similarly for exterior powers.

Now we can prove: 
\begin{thm}
  Suppose that $V$ is a weak complex variation of mixed   Hodge structure on $(X,D)$ with
  unipotent local monodromies around $D$. Then  $\wedge^a Gr_F^b(V)\otimes
  \det(F^{b+1})$ is nef for every $a$ and $b$. 
\end{thm}

\begin{proof}
Letting $Gr_F^{\max}V=F^{\max}V$ denote the smallest nonzero Hodge bundle of $V$
etc., we deduce from \eqref{eq:Sn} that 
$$ Gr_F^{\max} (S^n V) = S^n(Gr_F^{\max} V)  $$
Let $\omega_X=\Omega_X^{\dim X}$.
Applying theorem \ref{thm:main1}, to $S^nV$  shows that
$$H^i(S^n (Gr_F^{\max} V)\otimes \omega_X(D)\otimes L^m )=H^i(Gr_F^{\max} S^n V\otimes \omega_X(D)\otimes L^m )=0$$
for any $i>0$, $n>0$, $m>0$ and ample $L$. Therefore, by lemma
\ref{lemma:semipos}, $Gr_F^{\max} V$ is nef. 
            This is a special case of the theorem corresponding to $b={\max}$.

Fix $b$, let $r= \rank F^{b+1}$ and $0\le a\le \rank Gr_F^bV$. It is easy to see that 
\begin{equation*}
  \begin{split}
    Gr_F^{\max }(\wedge^{r+a} V) &\cong \wedge^a Gr_F^b(V)\otimes
\det(Gr_F^{b+1}(V))\otimes  \det(Gr_F^{b+2}(V))\ldots\\
&\cong \wedge^a Gr_F^b(V)\otimes
\det(F^{b+1}(V))
  \end{split}
\end{equation*}
Therefore this is nef by what was proved in the previous paragraph.

  \end{proof}

  \begin{rmk}\label{rmk:brune} Brunebarbe has pointed out to the
    author that the proof of  \cite[thm 3.4]{brunebarbe} carries over
    to complex variations without difficulty. In this form, it implies theorem \ref{thm:semipos}
   taking exterior powers as above. 
  \end{rmk}

\section{Proof of  theorem~\ref{thm:main2}}

We recall some notions from
\cite{arapura} needed below. Let $X$  be a smooth projective variety defined over
an algebraically closed field $k$. Suppose that $M$ is a vector bundle
on $X$. The Frobenius amplitude is a number $\phi(M)$
measuring the cohomological positivity in the following way.
If $\Char  k= p>0$, define the Frobenius power by $M^{(p^n)} = (Fr^n)^*M$. Then $\phi(M)$ is the
smallest natural number such that for any coherent sheaf $\mathcal{F}$
$$H^i(X,\mathcal{F}\otimes M^{(p^n)})=0$$
for $i>\phi(M)$ and all $n\gg 0$. If $\Char  k=0$, then we ``spread out''
$(X,M)$ over the spectrum $S$ of a finitely generated algebra, and take
the  supremum of $\phi(M_s)$ over almost all closed fibres. More
precisely $\phi(M) = \min\sup_{s\in U} \phi(M_s)$, as $U$ runs over
all nonempty opens of $S$.
So the smaller $\phi(M)$ is, the more positive $M$ is. This can be
related to more familiar notions of positivity.

\setcounter{thm}{\value{thmcount}}

\begin{thm}[{{\cite[thm 6.1]{arapura}}}]
  Suppose $\Char  k=0$. If $M$ is an ample vector bundle, then $\phi(M)<\rank M$.
\end{thm}

There is also a relative version of this. Suppose that $D\subset X$ is
a reduced effective divisor with normal crossings. In characteristic $p$
$$\phi(M,D) = \min \{\phi(M^{(p^n)}(-D'))\mid n\in \N, 0\le D'\le
(p^n-1)D\}$$
It is useful to observe that if the minimum is achieved for $n=n_0$,
then it is achieved for any $n\ge n_0$, because
$\phi(M^{(p^{n_0)}}(-D')) =\phi(M^{(p^{n_0+i})}(-p^iD'))$.
The definition  is extended to characteristic $0$ as above.

\begin{thm}[{{\cite[thm 6.7]{arapura}}}]\label{thm:arapura}
  Suppose $\Char  k=0$. If $M$ is a vector bundle,  such that $M(-\Delta)$ is ample
 for some fractional  effective $\Q$-divisor supported on $D$, then $\phi(M, D)<\rank M$.
\end{thm}

The proof given there does not work in positive
characteristic. However, we do get the same result for
line bundles by an easy argument.

\begin{lemma}\label{lemma:frob}
   Suppose that $\Char k$ is arbitrary. If $M$ is a line bundle,  such that $M(-\Delta)$ is ample
 for some fractional  effective $\Q$-divisor supported on $D$, then $\phi(M,D)=0$
\end{lemma}

\begin{proof}
As is well known, the ample divisors form an open cone in
$NS(X)\otimes \R$. Thus
for any small $\epsilon>0$, $M(-\Delta-\epsilon D)$ is still ample.
The set $S=\bigcup\frac{1}{p^n}\Z$ is clearly dense in
$\R$. Therefore  we can choose $\epsilon>0$ so that $M(-\Delta-\epsilon
D)$  is ample, and the coefficients of
$\Delta+\epsilon D$ are in $S\cap [0,1)$. Therefore
$M^{p^n}(-D')$ is ample for some integral divisor $0\le D'\le (p^n-1)D$.
The equality $\phi(M^{p^n}(-D'))=0$
   is an immediate consequence of Serre's vanishing theorem.
\end{proof}

Let $D'=\sum n_iD_i$ be a divisor supported on $D=\sum D_i$. 
If $f_i$ are the local equations of $D_i$, the formula
$$d(\prod_i f_i^{-n_i}) = \sum_j \prod_i f_i^{-n_i}\cdot( -n_j \frac{d
f_j} {f_j} )$$
shows that the fractional ideal
$\OO_X(D')\subset k(X)$ is stable under differentiation.
This rule endows  $\OO_X(D')$
 with an integrable connection that we denote $d^{D'}$. The above
 calculation shows that the residue of $d^{D'}$ on $D_i$ is $-n_i$.
If  $(V,\nabla)$ is a flat bundle on $(X,D)$, then $\nabla$ can be
extended to an operator on rational sections $V\otimes k(X)$.
This operation  restricts to the 
  product connection $\nabla\otimes d^{D'}$ on the subsheaf $V\otimes \OO_X(D')\subset V\otimes k(X)$.
 As above, we have
\begin{equation}
  \label{eq:ev}
Res_{D_i} \nabla\otimes d^{D'} = Res_{D_i} \nabla - n_i I  
\end{equation}
\cite[lemma 2.7]{ev2}. 
The next result is a generalization of \cite[lemma 3.3]{hara}. 

\begin{lemma}\label{lemma:hara}
  Suppose that $\Char k=p$ and that $0\le D'\le (p-1)D$ is a divisor.
If the residues of $\nabla$ are nilpotent, the natural inclusion 
$$\DR(V,\nabla)\subset \\DR(V\otimes \OO_X(D'),\nabla\otimes d^{D'})$$
is a quasiisomorphism.
\end{lemma}

\begin{proof}
Let $0=D'_0\subset D'_1\subset\ldots D'_N=D'$ be a sequence of
effective divisors such that the such $D'_{i+1}-D'_i$ consists of a
single component, say $D_{j(i)}$
Then \eqref{eq:ev} shows that $Res_{D_{j(i)}} \nabla\otimes
d^{D_{i+1}'}$ is invertible. Therefore \cite[lemma 2.10]{ev2} shows
that
$$\DR(V\otimes \OO(D'_i),  \nabla\otimes d^{D_i'} )\subset \DR(V\otimes \OO_X(D_{i+1}'),\nabla\otimes d^{D_{i+1}'})$$
is a quasiisomorphism.

\end{proof}

Given a Higgs bundle $(E,\theta)$, regarding $\theta\in
 H^0(\Omega_X^1(\log D)\otimes E^\vee\otimes E)$, we
get a dual Higgs field
$$E^\vee\stackrel{\theta^\vee}{\to} \Omega_X^1(\log D)\otimes
E^\vee$$

\begin{lemma}\label{lemma:duality}
   Suppose  that $M$ is   a vector bundle and that
$(E,\theta)$ Higgs bundle on $(X,D)$. If $(E,\theta)$ is nilpotent  of level at most
$\ell$, or semistable with   vanishing Chern classes, then $(E^\vee,\theta^\vee)$ has the
same properties. Furthermore
$$h^i(X,\DR(E,\theta)\otimes M)=
h^{2d-i}(X,\DR(E^\vee,\theta^\vee)(-D)\otimes M^\vee)$$
\end{lemma}

\begin{proof}
For the first statement, observe that if $E$ carries a $\theta$-stable
filtration $E= N^0\supset \ldots \supset N^\ell=0$ such that $Gr(\theta)=0$,
then $\ker[ E^\vee\to (N^i)^\vee]$ gives a filtration on the dual with
the same property. Suppose that $(E,\theta)$ is semistable with
vanishing Chern classes. Then $c_i(E^\vee)=\pm c_i(E)=0$. If $F\subset
E^\vee$ is a saturated $\theta^\vee$-stable subsheaf, then
$-\mu(F)=\mu(F^\vee)\ge \mu(E)=0$.

   Observe that 
$$RHom(\DR(E,\theta),\omega_X[d]) \cong
Hom(\DR(E,\theta),\omega_X[d])\cong {\DR}(E^\vee,\theta^\vee)(-D)[2d]$$
in the derived category of quasicoherent sheaves. Thus
Grothendieck-Serre duality  \cite{hartshorne} yields an isomorphism
$$H^i(X,\DR(E,\theta)\otimes M)=
H^{2d-i}(X,\DR(E^\vee,\theta^\vee)(-D)\otimes M^\vee)^\vee$$
\end{proof}

\begin{lemma}\label{lemma:boot1}
   Suppose that $\Char  k=p$, that $M$ is   a vector bundle and that
$(E,\theta)$ is a nilpotent semistable Higgs bundle of level less than
$p-d$.

\begin{enumerate}
\item If $0\le D'\le (p-1)D$, then
$$h^i(X, \DR(E,\theta)\otimes M)\le h^i(X,
\DR(B(E,\theta))\otimes M^{(p)}(D'))$$

\item If $0\le D'\le (p-1)D$, then
$$h^i(X, \DR(E,\theta)\otimes M(-D))\le h^i(X,
\DR(B(E,\theta))\otimes M^{(p)}(-D-D'))$$
\end{enumerate}

\end{lemma}

\begin{proof}
  By theorem \ref{thm:osv2}, lemma \ref{lemma:DRcart} and the projection formula
  \begin{equation*}
    \begin{split}
 H^i(X, \DR(E,\theta)\otimes M) &\cong H^i(X,
 (Fr_*\DR(C^{-1}(E,\theta)))\otimes M)     \\
&\cong H^i(X,
 Fr_*\DR(C^{-1}(E,\theta) \otimes (M^{(p)},\nabla_{cart}))) \\
&\cong H^i(X,
 \DR(C^{-1}(E,\theta) \otimes (M^{(p)},\nabla_{cart})))  \\
  \end{split}
  \end{equation*}
Since the residues of $C^{-1}(E,\theta) \otimes
(M^{(p)},\nabla_{cart})$ are nilpotent,
lemma \ref{lemma:hara} applies to show that this is isomorphic to
$$H^i(X,
 \DR(C^{-1}(E,\theta) \otimes (M^{(p)},\nabla_{cart}))(D'))  $$
  Furthermore, we obtain
$$h^i(X, \DR(C^{-1}(E,\theta) \otimes M^{(p)})(D'))  \le
h^i(X, \DR(\Lambda
C^{-1}(E,\theta))\otimes M^{(p)}(D'))$$
by lemma~\ref{lemma:Lambda}. This proves the first
 inequality.

The second inequality follows from the first using lemma~\ref{lemma:duality}.

\end{proof}

\begin{lemma}\label{lemma:boot2} 
   Suppose that  $0\le D'\le (p^n-1)D$, and with the remaining assumptions of the previous lemma.
 Then
$$h^i(X, \DR(E,\theta)\otimes M)\le h^i(X,
\DR(B^n(E,\theta))\otimes M^{(p^n)}(D'))$$
and 
$$h^i(X, \DR(E,\theta)\otimes M(-D))\le h^i(X, \DR(B^n(E,\theta))\otimes M^{(p^n)}(-D-D'))$$
\end{lemma}

\begin{proof}
  We may write $D' = p^{n-1} D_1' + p^{n-2} D_2'+\ldots$, where $0\le D_i'\le (p-1)D$. Then repeatedly 
applying lemma \ref{lemma:boot1} gives
\begin{equation*}
    \begin{split}
h^i(X, \DR(E,\theta)\otimes M) &\le h^i(X, \DR(B(E,\theta))\otimes M^{(p)}(D_1'))\\
&\le h^i(X, \DR(B^2(E,\theta))\otimes M^{(p^2)}(pD_1'+D_2'))\\
&\ldots
 \end{split}
  \end{equation*}
 This proves the first inequality. The second follows by duality.
\end{proof}

We are now ready to prove:
\setcounter{thm}{2}

\begin{thm}
With same assumptions as theorem \ref{thm:main1},
let  $M$ be a vector bundle when $\Char k=0$, or a line bundle when
$\Char k=p$, such that  $M(-\Delta)$ is ample
for some fractional  effective $\Q$-divisor supported on $D$.  
Then 
\begin{equation*}
 H^i(X,\DR(E,\theta)\otimes M(-D))=0  
\end{equation*}
for $i\ge d+ \rank M$.
\end{thm}

\begin{proof}
  This is  a modification of the proof  of theorem
  \ref{thm:main1}, so we just outline the main steps. 
As above, we first  reduce to the case where $k=
\overline{\mathbb{F}}_p$ with $p> d+\rank E$.
The sequence $(E_j,\theta_j)= B^j(E,\theta)$ is
eventually periodic. This fact 
 together with lemma \ref{lemma:boot2} implies that
$$h^i(X,\DR(E,\theta)\otimes M)\le h^i (X,\DR(E_n,\theta_n)(-D)\otimes (M^{(p^j)}(-D'))^{(p^\ell)})$$
for some $n,j$ and $0\le D'\le (p^j-1)D$ and $\ell$ 
arbitrarily large. Theorem \ref{thm:arapura} or lemma \ref{lemma:frob} shows that this is zero for $\ell$ large enough.
\end{proof}

\section{Proof of theorem \ref{thm:ks}}

Recall that we are given a complex projective variety $Y$,  a
 desingularization
 $\pi: X\to Y$ such that the exceptional divisor $D$ has simple normal
 crossings, and  a complex variation of Hodge structures  $\mathcal{V}$ on
  $U=X-D$ with quasiunipotent monodromies around $D$. 
   By Kawamata (\cite[thm 17]{kawamata} or \cite[lemma 3.19]{ev2}), we can find a finite Galois cover
   $p:Z\to X$ such that $Z$ is smooth, the branch divisor $D'\supseteq D$ has simple
   normal crossings, and $\mathcal{W}=p|_{p^{-1}U}^*\mathcal{V}$ has unipotent monodromies
   around $E= p^*D'_{red}$.  Since $V$  viewed as the filtered bundle
   with logarithmic connection is 
   determined by $\mathcal{V}|_{X-D'}$, we can and will replace $U$
   by $X - D'$ and $D$ by $D'$ with causing any harm.
Let $G= Gal(Z/X)$, and
  let $W$ denote the Deligne extension associated to
  $\mathcal{W}$.  By construction, $Z\to X$ is given as a tower of
  cyclic covers. It follows easily that $p^*\Omega_X^1(\log
  D)=\Omega_Z^1(\log E)$ cf. \cite[\S 3]{ev2}. Let $\nabla'$ denote the dual connection on
$W^\vee$. Then the composite
$$p_*W^\vee\stackrel{p_*\nabla'}{\longrightarrow} p_*(\Omega^1_Z(\log E)\otimes W^\vee)\cong \Omega^1_X(\log
D)\otimes p_*W^\vee$$
gives an induced logarithmic connection that we denote by  $\nabla$.

\begin{lemma}\label{lemma:res}
  The eigenvalues of the residues of  $\nabla$ are in $[0,1)$.
\end{lemma}

\begin{proof}
  The problem is local analytic, so we can reduce to the case where $Z\to X$
  is given by $z_i^{n_i} = x_i$, and $D$ by $x_1\ldots x_\ell
  =0$. Since the residues of $\nabla'$ are nilpotent, we can find a
  local basis $v_1,\ldots v_r$ of multivalued sections of  $W^\vee$
  such that 
$$\nabla'(v_j)
  \equiv 0\mod v_{j+1}, v_{j+2},\ldots$$ 
We can see that 
$$\{z_1^{m_1}\ldots z_d^{m_d}\otimes v_j\mid 0\le m_i < n_i, 1\le j\le
r\}$$
gives a basis of $p_*W^\vee$. We find that
$$\nabla (z_1^{m_1}\ldots z_d^{m_d}\otimes v_j)\equiv \sum_i
\frac{m_i}{n_i}z_1^{m_1}\ldots z_d^{m_d}\otimes v_j\otimes \frac{d
  x_i}{x_i}\mod v_{j+1},\ldots$$
\end{proof}

The group $G$ will act on  the pair $(p_*W^\vee,\nabla)$. 

  \begin{lemma}\label{lemma:SG}
\-
    \begin{enumerate}
    \item[(a)]    $(p_*W^\vee)^G = V^\vee$
\item[(b)] $p_*(\omega_{Z}\otimes W) ^G=\omega_X\otimes V$
\item[(c)]  $p_*(\omega_Z\otimes F^{\max} W)^G= S'(V)$
    \end{enumerate}
 
  \end{lemma}

  \begin{proof}
We can see easily that $(\pi_*\OO_Z)^G\cong \OO_X$, and that over $U$,
this is isomorphism is compatible with the direct image or Gauss-Manin
connection on the left and $d$ on the right.
Therefore by the
projection formula, it follows that $(\pi_*M)^G \cong M$ for any locally
free module. Furthermore, if $M|_U$ is equipped with a connection,
then the isomorphism  is compatible with it.
In particular,  $(p_*W^\vee)^G|_U = V^\vee|_U$ as
    flat vector bundles. By the previous lemma,  $(p_*W^\vee)^G$ gives
    the extension with eigenvalues of the residues in $[0,1)$. Since
    there is only one such extension \cite{deligne},  we must have
$(p_*W^\vee)^G = V^\vee$.

By Grothendieck duality \cite{hartshorne}, we have an isomorphism
\begin{equation}
  \label{eq:gd}
   (p_*W^\vee)^\vee\cong p_*(\omega_{Z/X}\otimes W)
\end{equation}
This isomorphism is canonical and therefore compatible with the
$G$-action. 
The  invariant part of a $G$-module $M$ is the image of the idempotent
$\frac{1}{|G|}\sum_g g$, and therefore it is the same as the
co-invariants. Thus
\begin{equation}
  \label{eq:Vvee}
(M^\vee)^G = (M^G)^\vee  
\end{equation}
Taking $G$-invariants of both sides of  \eqref{eq:gd}, and combining
this with  (a) and \eqref{eq:Vvee} yields $p_*(\omega_{Z/X}\otimes W)
^G=V$. This implies (b).

Let  $\tilde U = p^{-1}U$ and let $\tilde j:U\to Z$ denote the inclusion. Then $F^{\max} W=
{\tilde j}_*(F^{\max} \OO_{\tilde U}\otimes \mathcal{W})\cap W$ as subsheaves of
  ${\tilde j}_*(\OO_{\tilde U}\otimes \mathcal{W})$. This can also be expressed as
  the kernel of the difference of the inclusions
$$F^{\max} W=\ker [ {\tilde j}_*(F^{\max} \OO_{\tilde U}\otimes
\mathcal{W})\oplus  W\to {\tilde j}_*(\OO_{\tilde U}\otimes \mathcal{W})]$$
Therefore 
\begin{equation*}
   \begin{split}
     (\omega_Z\otimes F^{\max} W)^G &=\ker [ (\omega_Z\otimes {\tilde j}_*(F^{\max} \OO_{\tilde U}\otimes
 \mathcal{W}))^G\oplus (\omega_Z \otimes W)^G\to (\omega_Z\otimes {\tilde
   j}_*(\OO_{\tilde U}\otimes \mathcal{W}))^G]\\
&=\ker [ (\omega_X\otimes {j}_*(F^{\max} \OO_{ U}\otimes
 V))\oplus (\omega_X \otimes V)\to (\omega_X\otimes {
   j}_*(\OO_{U}\otimes V))]\\
 &= \omega_X\otimes F^{\max} V
   \end{split}
 \end{equation*}

  \end{proof}

\begin{thm}
Let $M$ be an ample vector bundle on $Y$, then
\begin{enumerate}
\item[(a)] 
$H^i(X,  S'(V)\otimes \pi^*M)=0$ 
for $i\ge \dim X +\rank M$.
\item[(b)] $R^i\pi_* (S'(V))=0$ for $i>0$.
\item[(c)] 
$H^i(Y, S( V)\otimes M)=0$
for $i\ge \dim X +\rank M$.
\end{enumerate}

\end{thm}

\begin{proof}
We can choose  a fractional effective $\Q$-divisor $\Delta$ supported on $D$ such that
$-\Delta$ is
relatively ample. Let $H$ be an ample  Cartier divisor on
$Y$. Then $\epsilon \pi^*H- \Delta$ is ample for every $\epsilon >0$.
We can choose $\epsilon >0$ so that $M(-\epsilon
H)$ remains  ample. Therefore $\pi^*M(-\epsilon H)$ is nef. Then
\cite[prop 6.2.11]{lazarsfeld} implies that
$\pi^*M(-\Delta)=\pi^*M(-\epsilon H+\epsilon H -\Delta)$ is ample.
Therefore $q^*M(-\Delta)$ is ample where $q= \pi\circ p$.
Theorem \ref{thm:main2} implies  that for $i\ge \dim X +\rank M$,
$$H^i(Z,  \omega_Z\otimes F^{\max} W\otimes q^*M)=0$$
Therefore
$$H^i(X, \omega_X\otimes F^{\max } V \otimes \pi^*M)=H^i(Z,  \omega_Z\otimes F^{\max} W
\otimes q^*M)^G=0$$
which proves (a).

Suppose that $M$ is an ample line bundle. By Serre, we can choose $N\gg 0$, so
that, $R^i\pi_* (\omega_X\otimes F^{\max } V)\otimes M^N$ is
globally generated for all $i$, and
$$H^j(Y, R^i\pi_* (\omega_X\otimes F^{\max } V)\otimes
M^N)=0,\quad  \forall i, \forall j>0$$
Therefore the Leray spectral sequence reduces to an isomorphism
$$H^i(X,\omega_X\otimes F^{\max } V\otimes
\pi^* M^N)\cong H^0(Y, R^i\pi_* (\omega_X\otimes F^{\max } V)\otimes
M^N)$$
By part (a), this must vanish for $i>0$. Therefore (b) holds.

Part (b) implies  that $S(V)\otimes M= \R \pi_* (\omega_X\otimes F^{\max}
 V\otimes \pi^* M)$. Therefore
$$H^i(Y, S( V)\otimes M)\cong H^i(X, \omega_X\otimes F^{\max}
 V\otimes \pi^* M)$$
By part (a) this vanishes for $i\ge \dim X +\rank M$.
\end{proof}

\section{Stationary Higgs bundles}

Let $(E,\theta)$ be a Higgs bundle on $(X,D)$, defined over $k$, which is semistable,
nilpotent with vanishing Chern classes. 
Let us say that $(E,\theta)$ is {\em stationary} if 
it is fixed under the bootstrapping operator $B=\Lambda C^{-1}$  (for  almost all
mod $p$ reductions). To be a bit more precise, let us first consider the case
where the ground field $k$ has characteristic $p>0$, and $(X,D)$ is
liftable mod $p^2$. Call $(E,\theta)$ stationary if it is nilpotent of
level at most $p$ and $B(E,\theta)\cong (E,\theta)$. If $\Char k=0$,
call $(E,\theta)$ stationary if everything can be spread out, as in
the proof of theorem~\ref{thm:main1}, such that the restriction of
$(E,\theta)$ to all the closed fibres are stationary. Note that  we say
``all the closed fibres'', because  the base $S$ may be shrunk. 

Given a Higgs bundle $(E,\theta)$, define the associated $i$th cohomological jump locus by
$$\Sigma^i(X; E,\theta) = \{L\in Pic^0(X)\mid
H^i(X,\DR(E,\theta)\otimes L)\not=0\}$$
When $(E,\theta)=(\OO_X,0)$ is trivial, an important structure theorem was proved by Green and Lazarsfeld \cite{gl} and
Simpson \cite{simpson2}. Pink and Roessler \cite{pr} showed how to
deduce these results from Deligne-Illusie.  Their argument generalizes to any stationary Higgs bundle.

\begin{thm}
  Suppose that $k$ is an algebraically closed field of characteristic
  $0$. If $(E,\theta)$ is stationary, then $\Sigma^i(X; E,\theta)$ is a
  finite union of translates of subabelian varieties of $Pic^0(X)$ by
  torsion elements.
\end{thm}

\begin{proof}
Arguing as in the proof of  \cite[thm 3.6]{pr}, we can reduce to the
case $k=\overline{\Q}$.  Then  $X,D, L, E, \theta$ will be defined over a
number field $k_0$.
We can spread   these out over $S\subset \Spec \OO_{k_0}$,
to obtain objects $\mathcal{X},\cD,\cL,\cE,\Theta$ as in the proof of theorem~\ref{thm:main1}.
Standard semicontinuity arguments show that $\Sigma^i(X;
E,\theta)$ is Zariski closed, and that it extends to a closed subscheme of 
$\cX$ with fibres $\Sigma^i(\cX_s;\cE_s, \Theta_s)$. Choose a closed
point $s\in S$ with residue characteristic $p\gg 0$.  By corollary
\ref{cor:osv} and the hypothesis
\begin{equation*}
  \begin{split}
    h^i(\cX_s, \DR(\cE_s,\Theta_s)\otimes \cL_s)  &\le
 h^i(\cX_s, \DR(B(\cE_s,\Theta_s))\otimes \cL_s^{p})\\
&= h^i(\cX_s, \DR(\cE_s,\Theta_s)\otimes \cL_s^{p})
  \end{split}
\end{equation*}
Therefore  $\Sigma^i(\cX_s;\cE_s, \Theta_s)$ is stable under
multiplication by $p$. Since this property holds for almost all $s\in S$, it follows from  \cite[thm 2.1]{pr} that this set is 
a  finite union of torsion translates of subabelian varieties.
\end{proof}

A natural question  asked by a referee is
whether $(E,\theta)$ is stationary when it arises from a polarizable
variation of Hodge structure with unipotent monodromy along $D$, as in
section 1. A positive answer would have interesting consequences, but
we are not very optimistic  in general. 
However, for variations of Hodge structure of geometric origin we do expect
this question to have a positive answer. Using \cite[thm 3.8]{ov}, it is not
difficult to prove:

\begin{prop}
  Let $f:Y\to X$ be a smooth projective map, and let  $E=\bigoplus
  R^af_*\Omega_{Y/X}^b$ be the equipped with the Higgs field  $\theta$ induced
  by the Kodaira-Spencer class. Then $(E,\theta)$ is stationary.
\end{prop}

We hope to pursue this question, and some consequences, in more detail elsewhere.

\end{document}